\documentclass{amsproc}%
\usepackage{amsfonts}
\usepackage{amsmath}
\usepackage{amssymb}
\usepackage{graphicx}
\usepackage{hyperref}%
\providecommand{\U}[1]{\protect\rule{.1in}{.1in}}
\theoremstyle{plain}
\newtheorem{acknowledgement}{Acknowledgement}

\newtheorem{corollary}{Corollary}

\newtheorem{definition}{Definition}
\newtheorem{example}{Example}

\newtheorem{lemma}{Lemma}

\newtheorem{proposition}{Proposition}
\newtheorem{remark}{Remark}

\newtheorem{theorem}{Theorem}
\numberwithin{equation}{section}
\begin{document}
\title[{\normalsize On Golomb topology of modules over commutative rings}]{{\normalsize On Golomb topology of modules over commutative rings}}
\author{U\u{g}ur Yi\u{g}it}
\address{Istanbul Medeniyet University, Department of Mathematics, 34700
\"{U}sk\"{u}dar, \.{I}stanbul, T\"{u}rk\.{I}ye.}
\email{ugur.yigit@medeniyet.edu.tr}
\author{Suat Ko\c{c}}
\address{Istanbul Medeniyet University, Department of Mathematics, 34700
\"{U}sk\"{u}dar, \.{I}stanbul, T\"{u}rk\.{I}ye.}
\email{suat.koc@medeniyet.edu.tr}
\author{\"{U}nsal Tekir}
\address{Marmara University, Department of Mathematics, 34720 Kad\i k\"{o}y,
\.{I}stanbul, T\"{u}rk\.{I}ye.}
\email{utekir@marmara.edu.tr}
\subjclass[2000]{13A15}
\keywords{Golomb topology, meet irreducible module, strongly irreducible submodule}

\begin{abstract}
In this paper, we associate a new topology to a nonzero unital module $M$ over
a commutative $R$, which is called Golomb topology of the $R$-module $M$. Let
$M\ $be an\ $R$-module and $B_{M}$ be the family of coprime cosets $\{m+N\}$
where $m\in M$ and $N\ $is a nonzero submodule of $M\ $such that $N+Rm=M$. We
prove that if $M\ $is a meet irreducible multiplication module or $M\ $is a
meet irreducible finitely generated module in which every maximal submodule is
strongly irreducible, then $B_{M}\ $is \ the basis for a topology on
$M\ $which is denoted by $\widetilde{G(M)}.$ In particular, the subspace
topology on $M-\{0\}$ is called the Golomb topology of the $R$-module $M\ $and
denoted by $G(M)$.\ We investigate the relations between topological
properties of $G(M)\ $and algebraic properties of $M.\ $In particular, we
characterize some important classes of modules such as simple modules,
Jacobson semisimple modules in terms of Golomb topology.

\end{abstract}
\maketitle

\section{Introduction}

Furstenberg topology on the set $\mathbb{Z}$ of integers is established by
Furstenberg \cite{Furs55} in attempt to give a topological proof of infinity
of prime numbers. It is defined by taking $\beta=\{a\mathbb{Z}+b|a,b\in
\mathbb{Z}\ \text{and}\ a>0\}$ as a basis. Golomb topology \cite{Golo59} on
the set $\mathbb{Z}_{+}$ of positive integers is constructed by taking the
collection $\beta=\{a\mathbb{N}+b|a,b>0\ $and$\ \gcd(a,b)=1\}$ as a basis. It
is defined to provide an example of a countably infinite connected Hausdorff
space. Kirch \cite{Kirc69} generalizes this by taking the collection
$\beta=\{p\mathbb{N}+b|b>0\ \text{and}\ p\ \text{prime not diving }b\}$ as a
subbase to give a countable, locally countable connected Hausdorff space.
Kirch topology is coarser than the Golomb topology, which means that every
open set in the Kirch topology is also open in the Golomb topology. An
analogous Golomb topology on any integral domain is established first by
Knopfmacher and Porubsk \cite{KP97} and subsequently developed by Clark
\cite{Clar17}. Clark \cite{Clar17} studies it to demonstrate that $R$ has
infinitely many nonassociate irreducibles, when $R$ is a semiprimitive domain
and not a field where every nonzero nonunit element has at least one
irreducible divisor. Some properties of these Golomb topologies is studied at
length by Clark et al. in \cite{CLP18}.

In this article, we give a generalization of the Golomb topology on modules.
In order to generalize the Golomb topology over modules, the features that a
module must provide are discussed. In the last section, we look into the
connections between the algebraic properties of a module $M$ and topological
properties of the Golomb topology $G(M)$ over module $M$. Specifically, we
describe several significant module classes in terms of Golomb topology,
including simple modules and Jacobson semisimple modules. Also, we provide an
example which shows that a module with the Golomb topology is not a
topological module.\bigskip

\section{The Golomb Topology on Modules}

In this paper, all rings under consideration are assumed to be commutative
with nonzero unity and all modules are nonzero unital. Let $R\ $always denote
such a ring and $M\ $denote such an $R$-module. For every submodule $N\ $of
$M,\ $the residual of $N\ $by $M\ $is defined as $(N:M)=\{r\in R:rM\subseteq
N\}$ \cite{Sharp}.\ Recall from \cite{Atani} that a submodule $N$ of $M\ $is
said to be a \textit{strongly irreducible} if whenever $K\cap L\subseteq N$
for some submodule $K,L\ $of $M,\ $then either $K\subseteq N$ or $L\subseteq
N.\ $An $R$-module $M\ $is said to be a \textit{meet-irreducible module} if
intersection of any two nonzero submodule is always nonzero \cite{Kasch}. Note
that $M\ $is meet-irreducible if and only if the zero submodule is a strongly
irreducible submodule of\ $M.\ $A proper submodule $N\ $of an $R$-module
$M\ $is said to be a \textit{prime submodule} if whenever $am\in N\ $for some
$a\in R$ and $m\in M,\ $then $a\in(N:M)\ $or $m\in N$ \cite{Sharp}. It is
known that if $N\ $is prime, then $(N:M)\ $is a prime ideal of $R.\ $Also it
is well-known that every prime ideal is strongly irreducible in a ring.
However, this is not true for modules over rings. See the following example.

\begin{example}
Consider $%
\mathbb{Z}
$-module $%
\mathbb{Z}
\oplus%
\mathbb{Z}
=M\ $and submodule $N=2
\mathbb{Z}
\oplus2
\mathbb{Z}
$.\ Then $N\ $is clearly a prime submodule of $M.\ $Choose, $K=
\mathbb{Z}
\oplus2
\mathbb{Z}
$ and $L=2
\mathbb{Z}
\oplus%
\mathbb{Z}
.\ $Since $K\cap L\subseteq N,\ K\nsubseteq N$ and $L\nsubseteq N,\ $it
follows that $N\ $is not a strongly irreducible submodule of $M.\ $
\end{example}

Recall from \cite{Barn} that an $R$-module $M\ $is said to be a
\textit{multiplication module} if every submodule $N\ $of $M$ has the form
$IM\ $for some ideal $I\ $of $R.\ $One can easily see that $M\ $is a
multiplication module if and only if $N=(N:M)M\ $for every submodule $N\ $of
$M\ $\cite{Smith}. Also, every prime submodule is a strongly irreducible
submodule in a multiplication module (See, \cite[Theorem 3.1]{Atani}). This is
not true for non multiplication modules. Indeed, a maximal submodule need not
to be a strongly irreducible. For instance, consider $3$-dimensional Euclidean
space $%
\mathbb{R}
^{3}$ and the maximal subspace $N=\{(x,y,0):x,y\in%
\mathbb{R}
\}.\ $Let $K=\{(x,x,x):x\in%
\mathbb{R}
\}$ and its orthogonal complement $K^{\perp}=\{(x,y,z):x+y+z=0\}.\ $Then note
that $K\cap K^{\perp}=\{0\}\subseteq N$, $K\nsubseteq N$ and $K^{\perp
}\nsubseteq N.\ $Thus, $N\ $is a non strongly irreducible maximal submodule of
$%
\mathbb{R}
^{3}.$\ Here, one can naturally ask "In which modules does a maximal submodule
imply strongly irreducible?" In the following, we will give a complete answer
for this question.

Let $R$ be a ring and $I,J$,$K$ are ideals of $R.$ It is well known that if
$I\ $is coprime to $J\ $and $K$, that is $I+J=R=I+K$, then $I\ $is coprime to
$J\cap K,$\ and this property can be generalized to finitely many ideals of
$R.\ $However, the following example shows that this is not true for general
modules over commutative rings.

\begin{example}
Consider $%
\mathbb{Z}
$-module $%
\mathbb{Z}
\oplus%
\mathbb{Z}
=M\ $and the submodules $N=\langle(1,1)\rangle$, $K=\langle(1,0)\rangle$ and
$L=\langle(0,1)\rangle\ $of $M.\ $Then one can easily see that $N+K=M=N+L$
\ and also $N+K\cap L=N\neq M.$\ Thus, $M\ $does not satisfy coprime condition
on its submodules.
\end{example}

\begin{definition}
An $R$-module $M\ $is said to satisfy finite coprime condition if whenever
$N+K_{i}=M$ for some submodules $N,K_{i}\ $of $M$ and $i=1,2,\ldots,n,\ $then
$N+
{\textstyle\bigcap\limits_{i=1}^{n}}
K_{i}=M.$
\end{definition}

\begin{lemma}
\label{lemfinitestar}An $R$-module $M\ $satisfies finite coprime condition if
and only if $N+K_{1}=M=N+K_{2}\ $for some submodules $N,K_{1},K_{2}$ of
$M\ $implies that $N+K_{1}\cap K_{2}=M.$
\end{lemma}

\begin{proof}
Is obvious.
\end{proof}

Now, we give the relations between finite coprime conditions and strongly
irreducible submodules with the following our first theorem.

\begin{theorem}
\label{finitestar}Let $M\ $be finitely generated $R$-module. The following
statements are equivalent. \ 

(i)\ $M\ $satisfies finite coprime condition.

(ii)\ If $N+K_{1}=M=N+K_{2}\ $for some submodules $N,K_{1},K_{2}$ of $M,$ then
$N+K_{1}\cap K_{2}=M.$

(iii) Every maximal submodule is strongly irreducible.
\end{theorem}

\begin{proof}
$(i)\Leftrightarrow(ii):\ $Follows from Lemma \ref{lemfinitestar}.

$(ii)\Rightarrow(iii):\ $Suppose that $(ii)\ $holds and $P\ $is a maximal
submodule of $M$. To prove that $P\ $is strongly irreducible, assume that
$K\cap L\subseteq P\ $for some submodules $K,L\ $of $M.\ $Let $K\nsubseteq
P\ $and $L\nsubseteq P.\ $Since $P$ is maximal, $P+K=M=P+L$ .\ Then by
$(ii),\ $we have $P+K\cap L=P=M\ $which is a contradiction. Thus, $K\subseteq
P$ or $L\subseteq P\ $which completes the proof.

$(iii)\Rightarrow(ii):\ $Assume that every maximal submodule is strongly
irreducible. Let $N+K_{1}=M=N+K_{2}\ $for some submodules $N,K_{1},K_{2}$ of
$M.\ $Now, we will show that $N+K_{1}\cap K_{2}=M.\ $Suppose that $N+K_{1}\cap
K_{2}\neq M.$ As $M\ $is finitely generated, by \cite{Lu}, there exists a
maximal submodule $P$ of $M\ $containing $N+K_{1}\cap K_{2}.\ $As $K_{1}\cap
K_{2}\subseteq P\ $and $P$ is maximal, by assumption, we have $K_{1}\subseteq
P$ or $K_{2}\subseteq P.\ $Without loss of generality, we may assume that
$K_{1}\subseteq P.\ $This implies that $N+K_{1}=M\subseteq P$ which is absurd.
Thus, we have the desired result $N+K_{1}\cap K_{2}=M.$
\end{proof}

Let $M\ $be an $R$-module. Then we define the family $\mathfrak{B}_{M}$ as the
set of all coprime cosets $\{m+N\}$ where $m\in M$ and $N\ $is a nonzero
submodule of $M\ $such that $N+Rm=M.\ $Let $R\ $be an integral domain. If we
consider $R\ $as an $R$-module, then Clark et al. showed that $\mathfrak{B}%
_{R}$ is the basis for a topology on $R\ $which is denoted by $\widetilde
{G(R)},\ $also the subspace topology of $\widetilde{G(R)}$ on $R^{\bullet
}=R-\{0\}$, denoted by $G(R),\ $is called the Golomb topology of $R$
\cite{CLP18}$.\ $Recall from \cite{Sharp} that an $R$-module $M$ is said to be
a \textit{torsion free module} if whenever $am=0$ for some $a\in R\ $and $m\in
M,\ $then $a=0\ $or $m=0.\ $It is clear that if $M\ $is a torsion free module,
then $R\ $must be integral domain. The class of torsion free modules can be
viewed as an extension of integral domains to the context of modules. So one
can naturally consider whether the family $\mathfrak{B}_{M}$ of coprime cosets
for torsion free $R$-module $M\ $is a basis for a topology on $M$ or not.\ The
following example illustrates that $\mathfrak{B}_{M}$ may not be a basis for
any topology on $M\ $even if $M\ $is a torsion free $R$-module.

\begin{example}
Consider $%
\mathbb{Z}
$-module $%
\mathbb{Z}
\oplus%
\mathbb{Z}
=M$ and the submodules $N_{1}=
\mathbb{Z}
\oplus0,\ N_{2}=0\oplus%
\mathbb{Z}
.\ $Let $m_{1}=(1,1)=m_{2}\in M.\ $Then $m_{1}+N_{1}$ and $m_{2}+N_{2}$ are in
$\mathfrak{B}_{M}.\ $Also, it is clear that $\left(  m_{1}+N_{1}\right)
\cap\left(  m_{2}+N_{2}\right)  =\{(1,1)\},$ and for any nonzero submodule
$K\ $of $M$, we have $(1,1)+K\nsubseteq\left(  m_{1}+N_{1}\right)  \cap\left(
m_{2}+N_{2}\right)  .\ $Thus, $\mathfrak{B}_{M}$ can not be basis for any
topology on $M.$
\end{example}

Now, in the following theorems, we prove that under some conditions on
$M,\ \mathfrak{B}_{M}$ is the basis for a topology on $M$ denoted by
$\widetilde{G(M)}.$ In this case, the subspace topology of $\widetilde{G(M)}$
on $M^{\bullet}=M-\{0\}$ is called Golomb topology of $M\ $and denoted by
$G(M).\ $

\begin{theorem}
\label{pGol}Let $M\ $be a meet irreducible finitely generated $R$-module over
which all maximal submodules are strongly irreducible, then $\mathfrak{B}_{M}$
is the basis for the topology $\widetilde{G(M)}.\ $
\end{theorem}

\begin{proof}
$(i):\ $Since $M\ $is finitely generated whose all maximal submodules are
strongly irreducible, by Theorem \ref{finitestar}, $M\ $satisfies finite
coprime condition. Let $(m_{1}+N_{1}),(m_{2}+N_{2})\in\mathfrak{B}_{M}$ and
choose $m\in(m_{1}+N_{1})\cap(m_{2}+N_{2}).\ $Then $N_{1},N_{2}\ $are nonzero,
and by meet irreducible condition, we have $N_{1}\cap N_{2}\neq0.$\ Since
$m_{i}+N_{i}\in\mathfrak{B}_{M},\ $we have $Rm_{i}+N_{i}=M$ for each
$i=1,2.$\ Also, since $m-m_{i}\in N_{i},\ $we have $Rm+N_{i}=M$ for each
$i=1,2.\ $By finite coprime condition, we conclude that $Rm+N_{1}\cap
N_{2}=M,\ $that is, $m+N_{1}\cap N_{2}\in\mathfrak{B}_{M}.\ $Also, it is clear
that $m+N_{1}\cap N_{2}\subseteq(m_{1}+N_{1})\cap(m_{2}+N_{2})$. Thus, we have
the desired result.
\end{proof}

\begin{theorem}
\label{pGol2}Let $M\ $be a multiplication\ meet irreducible $R$-module, then
$\mathfrak{B}_{M}$ is the basis for the topology $\widetilde{G(M)}.\ $
\end{theorem}

\begin{proof}
Suppose that $M\ $is a multiplication meet irreducible $R$-module. Let
$(m_{1}+N_{1}),(m_{2}+N_{2})\in\mathfrak{B}_{M}$ and choose $m\in(m_{1}%
+N_{1})\cap(m_{2}+N_{2}).\ $Then $N_{1},N_{2}\ $are nonzero, and by meet
irreducible condition, we have $N_{1}\cap N_{2}\neq0.$\ Since $m_{i}+N_{i}%
\in\mathfrak{B}_{M},\ $we have $Rm_{i}+N_{i}=M$ for each $i=1,2.$\ Also, since
$m-m_{i}\in N_{i},\ $we have $Rm+N_{i}=M$ for each $i=1,2.\ $Now, we will show
that $Rm+(N_{1}\cap N_{2})=M.\ $Suppose that $Rm+(N_{1}\cap N_{2})\neq
M.\ $Then by \cite[Theorem 2.5]{Smith}, there exists a maximal submodule $P$
of $M\ $containing $Rm+(N_{1}\cap N_{2}).\ $As $N_{1}\cap N_{2}\subseteq
P,\ $we have $(N_{1}\cap N_{2}:M)=(N_{1}:M)\cap(N_{2}:M)\subseteq(P:M).\ $As
$(P:M)$ is a prime ideal, we conclude that $(N_{1}:M)\subseteq(P:M)$ or
$(N_{2}:M)\subseteq(P:M).\ $Which implies that $N_{1}\subseteq P$ or
$N_{2}\subseteq P\ $as $M\ $is multiplication. Then we have $Rm+N_{1}%
=M\subseteq P$ or $Rm+N_{2}=M\subseteq P$ which is a contradiction. Thus
$Rm+N_{1}\cap N_{2}=M,\ $that is, $m+N_{1}\cap N_{2}\in\mathfrak{B}_{M}%
.\ $Also, $m+N_{1}\cap N_{2}\subseteq(m_{1}+N_{1})\cap(m_{2}+N_{2})$\ which
completes the proof.
\end{proof}

\section{Properties of the Golomb Topology on Modules}

\begin{theorem}
\label{tind}Let $M\ $be a meet irreducible finitely generated $R$-module whose
all maximal submodules are strongly irreducible or multiplication meet
irreducible $R$-module. The following statements are equivalent.

(i)\ $\widetilde{G(M)}\ $is an indiscrete space.

(ii)\ $G(M)\ $is an indiscrete space.

(iii) $M\ $is a simple module.
\end{theorem}

\begin{proof}
Assume that $M\ $is a meet irreducible finitely generated $R$-module over
which all maximal submodules are strongly irreducible.

$(i)\Rightarrow(ii):\ $Follows from the fact that a subspace of an indiscrete
space is also indiscrete.

$(ii)\Rightarrow(iii):\ $Suppose that $G(M)\ $is an indiscrete space. Now, we
will show that $M\ $is a simple module. Assume that $M\ $is not simple. Since
$M\ $is finitely generated, by \cite{Lu}, there exists a nonzero maximal
submodule $N\ $of $M.\ $Choose $m\in M-N.\ $Then we have $Rm+N=M,\ $that is
$m+N\in\mathfrak{B}_{M}.\ $As $\left(  m+N\right)  \cap M^{\bullet}=m+N$ and
$G(M)\ $is an indiscrete space, $m+N=M-\{0\}.\ $Now, we have two cases.
\textbf{Case 1: }Let $2m\neq0.\ $Since $2m\in M-\{0\}=m+N,\ $we conclude that
$2m-m=m\in N$ which is a contradiction. \textbf{Case 2:} Let $2m=0.\ $Since
$m+N=M-\{0\},\ $for every $0\neq x\in M,\ $there exists $n\in N\ $such that
$x=m+n.\ $Then we have $x+m=n\in N\ $which implies that $\{x+m:0\neq x\in
M\}=M-\{m\}\subseteq N.\ $As $N\ $is proper, we conclude that $M-\{m\}=N$.
This gives $M=N\cup\{m\}$ and so $M=N\cup Rm.$ As $N\neq M$, we conclude that
$Rm=M.\ $If $Rm=\{0,m\},\ $then we get $N=M-\{m\}=\{0\}$ which is absurd. So
we have $M=Rm\neq\{0,m\}.\ $Choose an$\ x\in R$ such that $xm\neq0,m.\ $Since
$M-\{m\}=N,\ $we have $xm\in N\ $which implies that $m-xm=(1-x)m\notin N.\ $On
the other hand, note that $m-xm=m$ since $N=M-\{m\}.$ This implies that $xm=0$
which is a contradiction. Therefore, $M\ $is a simple module.

$(iii)\Rightarrow(i):\ $Suppose that $M\ $is a simple module. Then only
nonzero submodule of $M\ $is $M\ $itself. Then for every $m\in M,\ $the only
neighbourhood of $m$ is $m+M=M$, that is, $\widetilde{G(M)}\ $is an indiscrete space.

The case "$M$ is multiplication meet irreducible $R$-module" is similar.
\end{proof}

Let $(X,\tau)$ be a topological space. A point $x\in X$ is said to be an
indiscrete point if the only open neighborhood of $x$ is $X.\ $It is clear
that $X$ is indiscrete if and only if every point of $X$ is indiscrete. We
know from Theorem \ref{tind} that if $M\ $is a simple module, then
$\widetilde{G(M)}\ $and $G(M)$ are indiscrete. So from now on, we assume that
$M\ $is not a simple module. Now, we investigate the indiscrete points of
$\widetilde{G(M)}$ and $G(M).$ Let $M\ $be an $R$-module and $N\ $be a
submodule of $M.\ $The Jacobson radical $J(N)\ $of $N\ $is defined as the
intersection of all maximal submodules of $M\ $containing $N.\ $In particular,
$J(M)=J((0)),\ $that is, intersection of all maximal submodules of $M.\ $If
$J(M)=(0),\ $then $M\ $is said to be a \textit{Jacobson semisimple}.

\begin{proposition}
\label{pind}Let $M\ $be finitely generated meet irreducible $R$-module
satisfying finite coprime condition or multiplication meet irreducible
$R$-module. The indiscrete points of \ $\widetilde{G(M)}$ (respectively,
$G(M))$ are the elements of $J(M)$ (respectively, $J(M)^{\bullet}).$
\end{proposition}

\begin{proof}
Without loss of generality, we may assume that $M$ is finitely generated meet
irreducible $R$-module satisfying finite coprime condition (The case "$M\ $is
multiplication meet irreducible $R$-module" is similar).

First, we will show that the indiscrete points of $\widetilde{G(M)}\ $is
$J(M).\ $Let $m\in J(M).\ $Choose a proper submodule $N$ of $M.\ $Since
$M\ $is finitely generated, there exists a maximal submodule $P\ $of
$M\ $containing $N.\ $Since $m\in J(M),\ $we have $m\in P.\ $This implies that
$Rm+N\subseteq P\neq M.$\ As $M\ $is not a simple module, we have $P\neq
0.\ $Thus, the only open neighborhood of $m$ is $m+M=M.\ $For the converse
assume that $m$ is an indiscrete point of $\widetilde{G(M)}.\ $Let $P$ be a
maximal submodule of $M.\ $Since $M\ $is not simple, we have $P\neq0.\ $If
$m\notin P,\ $then $Rm+P=M$ which implies that $m+P\ $is a neighborhood of
$m.\ $Note that $m+P\ $is nonempty and different from $M\ $which is a
contradiction. Hence, we have $m\in P,\ $that is, $m\in J(M).\ $Now, let $m\in
J(M)^{\bullet}.\ $This implies that $m$ is indiscrete point of $\widetilde
{G(M)},$ and thus it is an indiscrete point in $G(M).$ Now, assume that $m\in
M^{\bullet}$ is an indiscrete point of $G(M).$\ Let $P$ be a maximal submodule
of $M.\ $Assume that $m\notin P.\ $Then $m+P\ $is a neighborhood of $m$ in
$G(M)\ $since $0\notin m+P.\ $It is clear that $m+P$ is nonempty. A similar
argument in the proof of Theorem \ref{tind} shows that $m+P\neq M^{\bullet}$
and this is a contradiction. Thus, we have $m\in P,\ $namely, $m\in
J(M)^{\bullet}.$
\end{proof}

A topological space $X\ $is called a $T_{1}$-space or Fr\'{e}chet if for every
distinct points $x,y$ in $X,\ $there exist two open sets $O_{1},O_{2}\ $such
that $x\in O_{1},y\notin O_{1}$ and $y\in O_{2},x\notin O_{2}.\ $It is clear
that a topological space $X\ $is $T_{1}\ $if and only if $\overline
{\{x\}}=\{x\}$ for every point $x\in X.\ $

\begin{corollary}
Let $M\ $be finitely generated meet irreducible $R$-module satisfying finite
coprime condition or multiplication meet irreducible $R$-module. Then
$\widetilde{G(M)}$ is not $T_{1}$-space.
\end{corollary}

\begin{proof}
Without loss of generality, we may assume that $M$ is finitely generated meet
irreducible $R$-module satisfying finite coprime condition (The case "$M\ $is
multiplication meet irreducible $R$-module" is similar).

Since $M\neq0,\ $choose $0\neq m\in M.\ $As $0\in J(M),\ 0$ is indiscrete
point of $\widetilde{G(M)}.\ $This implies that $0\in\overline{\{m\}}$ and so
$\overline{\{m\}}\neq\{m\}.\ $Thus $M\ $is not a $T_{1}$-space.
\end{proof}

\begin{example}
Consider $\mathbb{Z} $-module $M= \mathbb{Z}_{8}$. Then $M$ is finitely
generated meet irreducible module satisfying finite coprime condition and
multiplication meet irreducible module. Then one can easily obtain
\begin{align*}
\mathfrak{B}_{M}  &  =\{ \mathbb{Z}_{8}, \overline{1}+\overline{2}\mathbb{Z}
_{8}, \overline{1}+\overline{4} \mathbb{Z}_{8}, \overline{3}+\overline{4}
\mathbb{Z}_{8}\}\\
&  =\{ \mathbb{Z}_{8}, \{\overline{1}, \overline{3}, \overline{5},
\overline{7}\}, \{\overline{1}, \overline{5}\}, \{\overline{3}, \overline
{7}\}\}
\end{align*}
and
\[
\widetilde{G(M)}=\{\emptyset, \mathbb{Z}_{8},\{\overline{1},\overline
{3},\overline{5},\overline{7}\},\{\overline{1},\overline{5}\},\{\overline
{3},\overline{7}\}\},
\]
and also
\[
G(M)=\{\emptyset,\mathbb{Z}_{8}-\{\overline{0}\},\{\overline{1},\overline
{3},\overline{5},\overline{7}\},\{\overline{1},\overline{5}\},\{\overline
{3},\overline{7}\}\}.
\]
It is clear that $\overline{\{\overline{2}\}}=\{\overline{0},\overline{2}
,\overline{4},\overline{6}\}\neq\{\overline{2}\}$ and hence $\widetilde{G(M)}$
is not a $T_{1}$-space.
\end{example}

\begin{remark}
An Abelian group $A$ is called a topological group if a topology $\tau_{A}$ is
defined on the set $A$ and the group operations
\begin{align*}
\mu:  &  A\times A\rightarrow A\ \ \ and\ \ \ \nu:A\rightarrow A\\
&  (a,b)\rightarrow a+b\ \ \ \ \ \ \ \ \ \ \ \ a\rightarrow-a
\end{align*}
are continuous. A ring $R$ is called a topological ring if a topology
$\tau_{R}$ is defined on the set $R$ and the additive group of the ring $R$ is
a topological group with this topology and the multiplication
\begin{align*}
\eta:  &  R\times R\rightarrow R\\
&  (a,b)\rightarrow a.b
\end{align*}
is also continuous. Similarly, an $R$-module $M$ is called a topological
$R$-module if there is a specified topology on $M$ such that $M$ is a
topological Abelian group and the scalar multiplication
\begin{align*}
\sigma:  &  R\times M\rightarrow M\\
&  (r,m)\rightarrow r.m
\end{align*}
is also continuous. In the previous example, $M=\mathbb{Z}_{8}$ and the
topology on it is $\widetilde{G(M)}=\{\emptyset,\mathbb{Z}_{8},\{\overline
{1},\overline{3},\overline{5},\overline{7}\},\{\overline{1},\overline
{5}\},\{\overline{3},\overline{7}\}\}$. $M=\mathbb{Z}_{8}$ with the topology
$\widetilde{G(M)}$ is not a topological module since the preimage
\begin{align*}
\mu^{-1}(\{\overline{1},\overline{5}\})=  &  \{(\overline{0},\overline
{1}),(\overline{1},\overline{0}),(\overline{2},\overline{7}),(\overline
{3},\overline{6}),(\overline{4},\overline{5}),(\overline{6},\overline
{3}),(\overline{7},\overline{2}),(\overline{0},\overline{5}),\\
&  (\overline{1},\overline{4}),(\overline{2},\overline{3}),(\overline
{3},\overline{2}),(\overline{4},\overline{1}),(\overline{5},\overline
{0}),(\overline{6},\overline{7}),(\overline{7},\overline{6})\}
\end{align*}
of $\{\overline{1},\overline{5}\}\in\widetilde{G(M)}$ is not an element of the
product topology $\widetilde{G(M)}\times\widetilde{G(M)}$ on $M\times M$,
where the basis for the product topology is $\beta_{M\times M}=\{\mathbb{Z}%
_{8}\times\mathbb{Z}_{8},\mathbb{Z}_{8}\times\{\overline{1},\overline
{5}\},\mathbb{Z}_{8}\times\{\overline{3},\overline{7}\},\{\overline
{1},\overline{5}\}\times\mathbb{Z}_{8},\{\overline{3},\overline{7}%
\}\times\mathbb{Z}_{8}\}$.
\end{remark}

Let $K,N\ $be two proper submodules of $M$. In this case, it is clear that
$(N:M)+(K:M)\subseteq(N+K:M)$ and the equality does not hold in general. For
instance, consider $%
\mathbb{Z}
$-module $M=%
\mathbb{Z}
\oplus%
\mathbb{Z}
\ $and the submodules $N=0\oplus%
\mathbb{Z}
$ and $K=%
\mathbb{Z}
\oplus0$ of $M.\ $Then note that $(N:M)=(K:M)=(0)\ $and also $(N+K:M)=%
\mathbb{Z}
$.\ Recall from \cite{Smith2} that an $R$-module $M\ $is said to be a $\mu
$-module if $(N+K:M)=(N:M)+(K:M)$ for every submodules $N,K\ $of $M.\ $Every
finitely generated multiplication module is an example of $\mu$-module. The
authors in \cite[Theorem 5]{UlTeKo} proved that the Chinese Remainder Theorem
is valid for $\mu$-modules. For the sake of completeness, we will give the
Chinese Remainder Theorem for $\mu$-modules in the following.

\begin{theorem}
\label{tchi}\textbf{(Chinese Remainder Theorem) }Let $M\ $be a $\mu$-module
and $K,N\ $be two submodules of $M\ $such that $K+N=M.\ $Consider
$R$-homomorphism $\pi:M\rightarrow M/K\oplus M/N$ defined by $\pi
(m)=(m+K,m+N)$ for each $m\in M.\ $Then $\pi$ is surjective and $Ker\pi=K\cap
N.\ $In this case, $M/K\cap N\cong M/K\oplus M/N.$
\end{theorem}

\begin{proof}
Let $K+N=M.\ $As $M\ $is a $\mu$-module, we have
$(M:M)=R=(K+N:M)=(K:M)+(N:M).$ Then we can write $1=a+b$ for some $a\in(K:M)$
and $b\in(N:M).\ $Now, let $(x+K,y+N)\in M/K\oplus M/N.\ $Put $z=bx+ay\in
M.\ $In this case, we have $z+K=bx+ay+K=bx+K=bx+ax+K=x+K.\ $Likewise, we have
$z+N=y+N.\ $Thus, we obtain $\pi(z)=(z+K,z+N)=(x+K,y+N),\ $that is, $\pi$ is
surjective. Also, it is clear that $Ker\pi=K\cap N\ $and the rest follows from
the First Isomorphism Theorem.
\end{proof}

\begin{corollary}
\label{cchi}Let $M\ $be a $\mu$-module and $K,N\ $be two submodules of
$M\ $such that $K+N=M.$\ Then for every $x,y\in M,\ $we have $(x+N)\cap
(y+K)\neq\emptyset.$
\end{corollary}

\begin{proof}
Let $x,y\in M$ and $K+N=M,\ $where $M\ $is a $\mu$-module and $K,N\ $are
submodules of $M.\ $Then by Theorem \ref{tchi}, $\pi:M\rightarrow M/K\oplus
M/N$ is surjective homomorphism. In this case, $\pi(m)=(m+K,m+N)=(y+K,x+N)\ $
for some $m\in M.\ $Thus we have $m\in(x+N)\cap(y+K).\ $
\end{proof}

\begin{theorem}
\label{tclo}Let $M\ $be finitely generated meet irreducible $\mu$-module
satisfying finite coprime condition or multiplication meet irreducible $\mu
$-module.\ Suppose that $N$ is a submodule of $M\ $and $m\in M.\ $The
following statements are satisfied.

(i)\ $J(N)\subseteq\overline{m+N}$ in $\widetilde{G(M)}.$

(ii)\ $\overline{N}=J(N)$ in $\widetilde{G(M)}.$

(iii)\ Suppose that $m\notin N.\ $Then we have $J(N)^{\bullet}\subseteq
\overline{(m+N)^{\bullet}}.$
\end{theorem}

\begin{proof}
Without loss of generality, we may assume that $M\ $is a finitely generated
meet irreducible $\mu$-module satisfying finite coprime condition (The case
"$M\ $is multiplication meet irreducible $\mu$-module" is similar).

$(i):\ $Let $x\in J(N)\ $and choose a neighborhood $O\ $of $x.\ $Then there
exists a coprime coset $z+P\ $such that $x\in z+P\subseteq O.\ $In this case,
we have $x+P=z+P\subseteq O$.$\ $Now, we will show that $N+P=M.\ $Assume that
$N+P\neq M.\ $As $M\ $is finitely generated, there exists a maximal submodule
$Q\ $of $M$ such that $N+P\subseteq Q.\ $Since $x\in J(N)\ $and $N\subseteq
Q,\ $we have $x\in Q$ which implies that $Rx+P=M\subseteq Q$. This is a
contradiction. Thus, we have $N+P=M.$\ Since $M\ $is a $\mu$-module, by
Corollary \ref{cchi}, we have $x+P\ $and $m+N$ intersect. Then we have $x\in$
$\overline{m+N}.$

$(ii):\ $The inclusion $J(N)\subseteq\overline{N}$ follows from (i). For the
converse, we note that the condition "$\mu$-module" is not necessary. In fact,
choose $m\in\overline{N}$\ and a maximal submodule $P$ of $M\ $containing
$N.\ $Assume that $m\notin P.\ $As $P$ is maximal, $m+P\ $is a coprime coset
containing $m.\ $On the other hand, note that $(m+P)\cap P=\emptyset$ which
implies that $(m+P)\cap N=\emptyset.\ $This gives a contradiction. Thus we
have $m\in P,\ $that is, $m\in J(N).$

$(iii):\ $Since $m\notin N,\ $we have $0\notin m+N$, that is, $m+N\subseteq
M^{\bullet}.\ $The rest follows from (i).
\end{proof}

\begin{corollary}
\label{czeroclosure}Let $M$ be a finitely generated meet irreducible
$R$-module satisfying finite coprime condition or multiplication meet
irreducible $R$-module. Then we have $\overline{\{0\}}=J(M)$ in $\widetilde
{G(M)}.$
\end{corollary}

\begin{proof}
Follows from Theorem \ref{tclo} $(ii)$ and the fact that $J(M)=J((0))$.
\end{proof}

\begin{theorem}
Let $M$ be a multiplication $R$-module and $ann(M)$ is a prime ideal of
$R.\ $The following statements are equivalent.

(i) $M\ $is Jacobson semisimple, that is, $J(M)=0.$

(ii)\ $G(M)\ $is a $T_{2}$-space.

(iii)\ $G(M)\ $is a $T_{1}$-space.

(iv) $G(M)\ $is a $T_{0}$-space.

(v) $\widetilde{G(M)}\ $is a $T_{0}$-space.

(vi) In $\widetilde{G(M)},\ $we have $\overline{\{0\}}=\{0\}.$
\end{theorem}

\begin{proof}
Assume that $M$ is a multiplication $R$-module and $ann(M)$ is a prime ideal
of $R.\ $Let $K\cap N=0\ $for some submodules $K,N\ $of $M.\ $Then $(K\cap
N:M)=(K:M)\cap(N:M)=ann(M).\ $As $ann(M)\ $is a prime ideal, we have
$(K:M)=ann(M)\ $or $(N:M)=ann(M).\ $As $M\ $is a multiplication module, we
have $K=0\ $or $N=0,\ $that is, $M\ $is meet irreducible. Then by Proposition
\ref{pGol2}, we have the Golomb topology.

$(i)\Rightarrow(ii):\ $Suppose that $M\ $is Jacobson semisimple, that is,
$J(M)=0.$\ Take $m\neq n\in M^{\bullet}.\ $Since $M\ $is multiplication,
$(Rm:M),\ (Rn:M)\ $and $(R(m-n):M)$ are nonzero ideals. Also note that
$(Rm:M)(Rn:M)R(m-n)\neq0$ since $ann(M)\ $is a prime ideal and $M$ is
multiplication. Since $J(M)=0,\ $there exists a maximal submodule $P$ of
$M\ $such that $(Rm:M)(Rn:M)R(m-n)\nsubseteq P.\ $As $m,n\notin P,\ m+P\ $and
$n+P$ are coprime cosets. Since $m-n\notin P,\ $we have $(m+P)\cap
(n+P)=\emptyset.\ $Thus, $G(M)\ $is a $T_{2}$-space.

$(ii)\Rightarrow(iii)\Rightarrow(iv):$ It is clear.

$(iv)\Rightarrow(i):\ $Suppose that $G(M)\ $is a $T_{0}$-space. Let
$J(M)\neq0.\ $Choose a nonzero element $m\in J(M).\ $Since $(Rm:M)\neq
ann(M),\ (Rm:M)/ann(M)$ is a nonzero ideal of the integral domain
$R/ann(M).\ $Thus $(Rm:M)/ann(M)$ contains infinitely many elements. Let
$(Rm:M)/ann(M)=\{x_{i}+ann(M):x_{i}\in R\}_{i\in\Delta}.$\ Choose two distinct
nonzero element $x_{i}+ann(M),x_{j}+ann(M)\in(Rm:M)/ann(M)$ such that
$x_{i}+ann(M)\neq-x_{j}+ann(M).\ $Since $x_{i},x_{j}\in(Rm:M),\ $note that
$x_{i}M$ and $x_{j}M$ are contained in $Rm$. Now, we will show that $x_{i}%
^{2}m\neq x_{j}^{2}m.\ $Assume that $x_{i}^{2}m=x_{j}^{2}m.\ $Then we have
$(x_{i} ^{2}-x_{j}^{2})m=0\ $which implies that $(x_{i}^{2}-x_{j}^{2})x_{i}
M\subseteq(x_{i}^{2}-x_{j}^{2})Rm=0.\ $Thus, we conclude that $(x_{i}
^{2}-x_{j}^{2})x_{i}\in ann(M).\ $As $ann(M)$ is a prime ideal and
$x_{i}\notin ann(M)$, we have $x_{i}-x_{j}\in ann(M)$ or $x_{i}+x_{j}\in
ann(M).\ $Thus, we obtain $x_{i}+ann(M)=x_{j}+ann(M)$ or $x_{i}+ann(M)=-x_{j}
+ann(M)$ which is a contradiction. Thus we have $x_{1}^{2}m\neq x_{2}^{2}m\in
J(M)^{\bullet}.\ $Then by Proposition \ref{pind}, $x_{1}^{2}m\neq x_{2}^{2}m$
are indiscrete points of $G(M).\ $Thus, $G(M)\ $is not a $T_{0}$-space which
is a contradiction. Hence, $J(M)=0.\ $

$(v)\Rightarrow(vi):\ $Let $\widetilde{G(M)}\ $be a $T_{0}$-space. Then
$G(M)\ $is a $T_{0}$-space, by $(iv)\Leftrightarrow(i),\ $we have
$J(M)=0.\ $Again by Corollary \ref{czeroclosure}, we have $\overline
{\{0\}}=J(M)=\{0\}\ $which completes the proof.

$(vi)\Rightarrow(i):\ $Suppose that $\overline{\{0\}}=\{0\}$ in $\widetilde
{G(M)}.\ $By Corollary \ref{czeroclosure}, we know that $\overline
{\{0\}}=J(M).$ The rest is clear.

$(i)\Rightarrow(v):\ $Let $J(M)=0\ $and $m\neq n\in M.\ $If $m,n\neq
0,\ $similar argument in $(i)\Rightarrow(ii)\ $shows that $(m+P)\cap
(n+P)=\emptyset\ $for some maximal submodule $P$ of $M.\ $Now, assume that
$m=0,n\neq0.\ $Since $n\notin J(M),\ n\ $is not indiscrete point of
$\widetilde{G(M)}.\ $Thus there exists an open set $O\neq M$ such that $n\in
O$.\ Note that $0=m\notin O$ since $0=m$ is indiscrete point. Hence,
$\widetilde{G(M)}\ $is a $T_{0}$-space.
\end{proof}

\begin{acknowledgement} On behalf of all authors, the corresponding author states that there is no conflict of interest.
\end{acknowledgement}

\end{document}